\documentclass[a4paper,11pt]{amsart}
\usepackage{amsmath,amssymb,amsfonts,amsthm}

\usepackage[alphabetic]{amsrefs}
\usepackage{bm}

\usepackage{graphicx}
\usepackage{ascmac}
\usepackage[all]{xy}
\usepackage{amsthm}
\usepackage[top=25mm,bottom=25mm,left=25mm,right=25mm]{geometry}
\usepackage{tikz}
\usetikzlibrary{cd,arrows,matrix,backgrounds,shapes,positioning,calc,decorations.markings,decorations.pathmorphing,decorations.pathreplacing}
\tikzset{blackv/.style={circle,fill=black,inner sep=3pt,outer sep=3pt},
         whitev/.style={circle,fill=white,draw=black,inner sep=3pt,outer sep=3pt},
         blabel/.style={circle,draw=black,inner sep=1.5pt,outer sep=0pt},
         redv/.style={circle,fill=red,inner sep=3pt,outer sep=3pt},
         block/.style={draw,rectangle split,rectangle split horizontal,rectangle split parts=#1},
         symbol/.style={
           draw=none,
           every to/.append style={
             edge node={node [sloped, allow upside down, auto=false]{$#1$}}}}
}
\usepackage{pxpgfmark}
\usepackage{multirow}
\usetikzlibrary{intersections,arrows,calc,decorations.markings}
\newtheorem{theorem}{Theorem}[section]

\newtheorem{proposition}[theorem]{Proposition}

\newtheorem{corollary}[theorem]{Corollary}
\newtheorem{lemma}[theorem]{Lemma}

\theoremstyle{definition}
\newtheorem{remark}[theorem]{Remark}

\newtheorem{condition}[theorem]{Condition}

\newtheorem{question}[theorem]{Question}

\makeatletter
 
 \@addtoreset{equation}{section}
\makeatother

\def\xx{\mathbf{x}}

\def\TT{\mathbb{T}}

\def\ZZ{\mathbb{Z}}

\def\Fcal{\mathcal{F}}

\title[Generalization of Markov Diophantine equation via generalized cluster algebra]{Generalization of Markov Diophantine equation\\via generalized cluster algebra}
\author{Yasuaki Gyoda}
\author{Kodai Matsushita}
\address{(Yasuaki Gyoda) Graduate School of Mathematical Sciences, The University of Tokyo, 3-8-1 Komaba Meguro-ku Tokyo 153-8914, Japan}
\email{m17009g@math.nagoya-u.ac.jp}
\address{(Kodai Matsushita) Graduate School of Mathematics, Nagoya University, Chikusa-ku, Nagoya, 464-8602 Japan}
\email{m19043g@math.nagoya-u.ac.jp}
\begin{document}
\begin{abstract}
In this paper, we deal with two classes of Diophantine equations, $x^2+y^2+z^2+k_1yz+k_2zx+k_3xy=(3+k_1+k_2+k_3)xyz$ and $x^2+y^4+z^4+ky^2z^2+2xz^2+2xy^2=(7+k)xy^2z^2$, where $k_1,k_2,k_3,k$ are nonnegative integers. The former is known as the Markov Diophantine equation if $k_1=k_2=k_3=0$, and the latter is a Diophantine equation recently studied by Lampe if $k=0$. We give algorithms to enumerate all positive integer solutions to these equations, and discuss the structures of the generalized cluster algebras behind them.
\end{abstract}
\maketitle
\section{Introduction}
In this paper, we deal with some Diophantine equations. One of equations with which we deal has the following form:
\begin{align}\label{Diophantine}
    x^2+y^2+z^2+k_1yz+k_2zx+k_3xy&=(3+k_1+k_2+k_3)xyz,
\end{align}
where $k_1,k_2,k_3 \in \ZZ_{\geq 0}$.
We describe all positive integer solutions to \eqref{Diophantine} in a combinatorial way. We give a tree $\mathbb T^{k_1,k_2,k_3}$ with triplets of positive integers as its vertices in the following steps.

\begin{itemize}
    \item [(1)] The root vertex is $(1,1,1)$,
    \item[(2)] the triplet $(1,1,1)$ has three children, $(k_1+2,1,1),(1,k_2+2,1),(1,1,k_3+2),$
    \item [(3)] the generation rule below $(k_1+2,1,1),(1,k_2+2,1),(1,1,k_3+2)$ is as follows: 
\begin{itemize}
        \item [(i)] if $a$ is a maximal number in $(a,b,c)$, then $(a,b,c)$ has two children 
        \[\left(a,\dfrac{a^2+k_2ac+c^2}{b},c\right)\text{ and } \left(a,b,\dfrac{a^2+k_3ab+b^2}{c}\right),\] 
        \item [(ii)] if $b$ is a maximal number in $(a,b,c)$, then $(a,b,c)$ has two children
        \[\left(\dfrac{b^2+k_1bc+c^2}{a},b,c\right)\text{ and } \left(a,b,\dfrac{a^2+k_3ab+b^2}{c}\right),\] 
        \item [(iii)] if $c$ is a maximal number in $(a,b,c)$, then $(a,b,c)$ has two children 
        \[\left(\dfrac{b^2+k_1bc+c^2}{a},b,c\right)\text{ and } \left(a,\dfrac{a^2+k_2ac+c^2}{b},c\right).\] 
\end{itemize}
\end{itemize}
When $k_1=1,k_2=2,k_3=0$, the first few terms of $\TT^{1,2,0}$ are as follows:
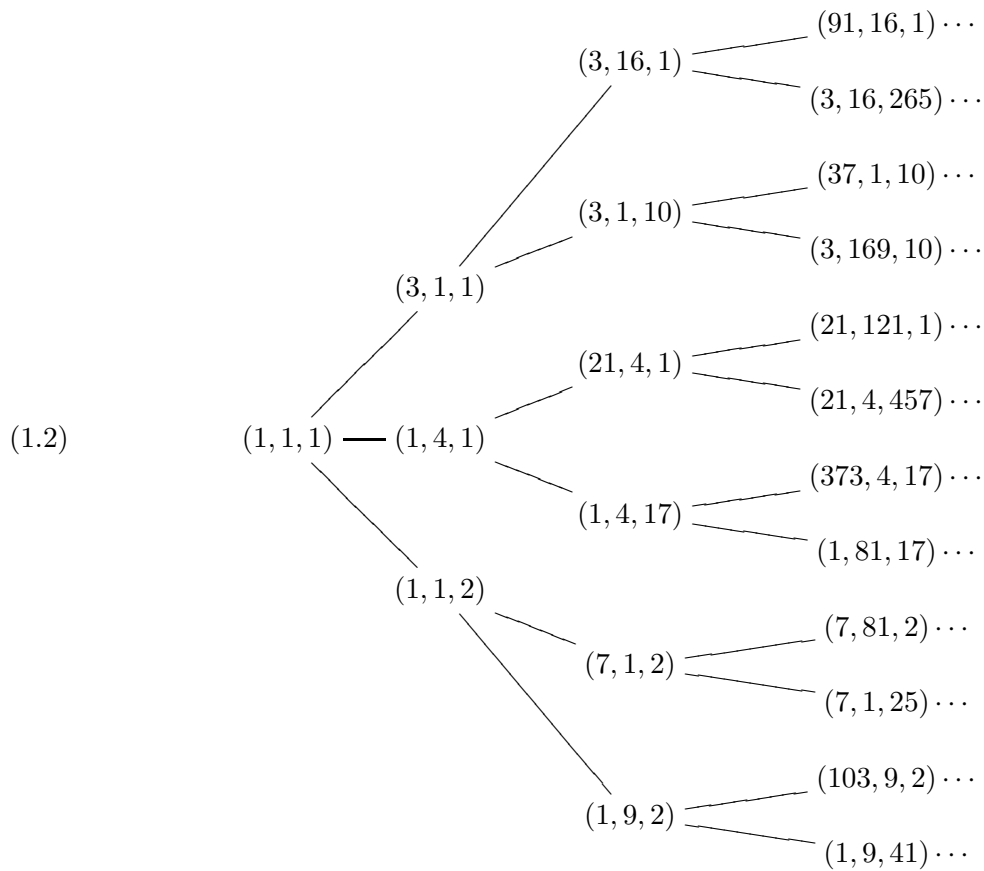
\begin{align}\label{tree}
\begin{xy}(0,0)*+{(1,1,1)}="0",(20,20)*+{(3,1,1)}="1",(20,0)*+{(1,4,1)}="1'",(20,-20)*+{(1,1,2)}="1''",(45,50)*+{(3,16,1)}="20",(45,30)*+{(3,1,10)}="21",(45,10)*+{(21,4,1)}="22",(45,-10)*+{(1,4,17)}="23",(45,-30)*+{(7,1,2)}="24",(45,-50)*+{(1,9,2)}="25",(80,55)*+{(91,16,1)\cdots}="40",(80,45)*+{(3,16,265)\cdots}="41", (80,35)*+{(37,1,10)\cdots}="42", (80,25)*+{(3,169,10)\cdots}="43", (80,15)*+{(21,121,1)\cdots}="44", (80,5)*+{(21,4,457)\cdots}="45", (80,-5)*+{(373,4,17)\cdots}="46", (80,-15)*+{(1,81,17)\cdots}="47", (80,-25)*+{(7,81,2)\cdots}="48", (80,-35)*+{(7,1,25)\cdots}="49", (80,-45)*+{(103,9,2)\cdots}="410", (80,-55)*+{(1,9,41)\cdots}="411", \ar@{-}"0";"1"\ar@{-}"0";"1'"\ar@{-}"0";"1''"\ar@{-}"1";"20"\ar@{-}"1";"21"\ar@{-}"1'";"22"\ar@{-}"1'";"23"\ar@{-}"1''";"24"\ar@{-}"1''";"25"\ar@{-}"20";"40"\ar@{-}"20";"41"\ar@{-}"21";"42"\ar@{-}"21";"43"\ar@{-}"22";"44"\ar@{-}"22";"45"\ar@{-}"23";"46"\ar@{-}"23";"47"\ar@{-}"24";"48"\ar@{-}"24";"49"\ar@{-}"25";"410"\ar@{-}"25";"411"
\end{xy}.
\end{align}

The first main result is the following theorem:

\begin{theorem}\label{Diophantinetheorem}
Every positive integer solution to \eqref{Diophantine} appears exactly once in $\mathbb T^{k_1,k_2,k_3}$.
\end{theorem}

When $k_1=k_2=k_3=0$, the equation \eqref{Diophantine} is the \emph{Markov Diophantine equation}
\begin{align}\label{Diophantine2}
    x^2+y^2+z^2=3xyz.
\end{align}
This is an equation that has received much attention since the work on the \emph{Markov spectrum}, and is now being studied in relation to combinatorial objects such as \emph{Christoffel words}, \emph{perfect matchings} of graphs, and \emph{continuous fractions} (for detail, see Aigner's book \cite{aig}). The proof of Theorem \ref{Diophantinetheorem} for the case of $k_1=k_2=k_3=0$ is known, for example, by \cite{aig}*{Section 3.1}.

Moreover, when $k_1=k_2=k_3=1$, the equation \eqref{Diophantine} is a Diophantine equation
\begin{align}\label{Diophantine3}
    (x+y)^2+(y+z)^2+(z+x)^2=12xyz.
\end{align}
studied in \cite{gyo21}. The positive integer solutions to this equation, as well as the Markov equation, have been shown to be closely related to perfect matchings of graphs and continuous fractions. The specialized version of Theorem \ref{Diophantinetheorem} for the case of $k_1=k_2=k_3=1$ is proved by \cite{gyo21}*{Theorem 1.1}.

Furthermore, Lampe proved specialized version of Theorem \ref{Diophantinetheorem} for the case of $k_1=0,k_2=k_3=2$ in \cite{lampe}*{Lemma 2.7}, that is, the description of all positive integer solutions to 
\begin{align}\label{Diophantine7}
    x^2+y^2+z^2+2xy+2zx=7xyz.
\end{align}
In \cite{lampe}, this theorem is used to describe all positive integer solutions to 
\begin{align}\label{Diophantine8}
    x^2+y^4+z^4+2zx^2+2xy^2=7xy^2z^2.
\end{align}

In this paper, we also deal with the generalized version of the equation \eqref{Diophantine8},
\begin{align}\label{Diophantine4}
    x^2+y^4+z^4+ky^2z^2+2xz^2+2xy^2=(7+k)xy^2z^2,
\end{align}
where $k\in \ZZ_{\geq 0}$.

As in \eqref{Diophantine}, we describe all positive integer solutions to \eqref{Diophantine4} in a combinatorial way. We give a tree $\mathbb T^k$ with triplets of positive integers as its vertices in the following steps.

\begin{itemize}
    \item [(1)] The root vertex is $(1,1,1)$,
    \item[(2)] the triplet $(1,1,1)$ has three children, $(k+2,1,1),(1,2,1),(1,1,2),$
    \item [(3)] the generation rule below $(k+2,1,1),(1,2,1),(1,1,2)$ is as follows:
\begin{itemize}
        \item [(i)] if $a$ is a maximal number in $(a,b^2,c^2)$, then $(a,b,c)$ has two children 
        \[\left(a,\dfrac{a+c^2}{b},c\right)\text{ and } \left(a,b,\dfrac{a+b^2}{c}\right),\] 
        \item [(ii)] if $b^2$ is a maximal number in $(a,b^2,c^2)$, then $(a,b,c)$ has two children
        \[\left(\dfrac{b^4+kb^2c^2+c^4}{a},b,c\right)\text{ and } \left(a,b,\dfrac{a+b^2}{c}\right),\] 
        \item [(iii)] if $c^2$ is a maximal number in $(a,b^2,c^2)$, then $(a,b,c)$ has two children 
        \[\left(\dfrac{b^4+kb^2c^2+c^4}{a},b,c\right)\text{ and } \left(a,\dfrac{a+c^2}{b},c\right).\]  
\end{itemize}
\end{itemize}
When $k=1$, the first few terms of $\TT^1$ are as follows:
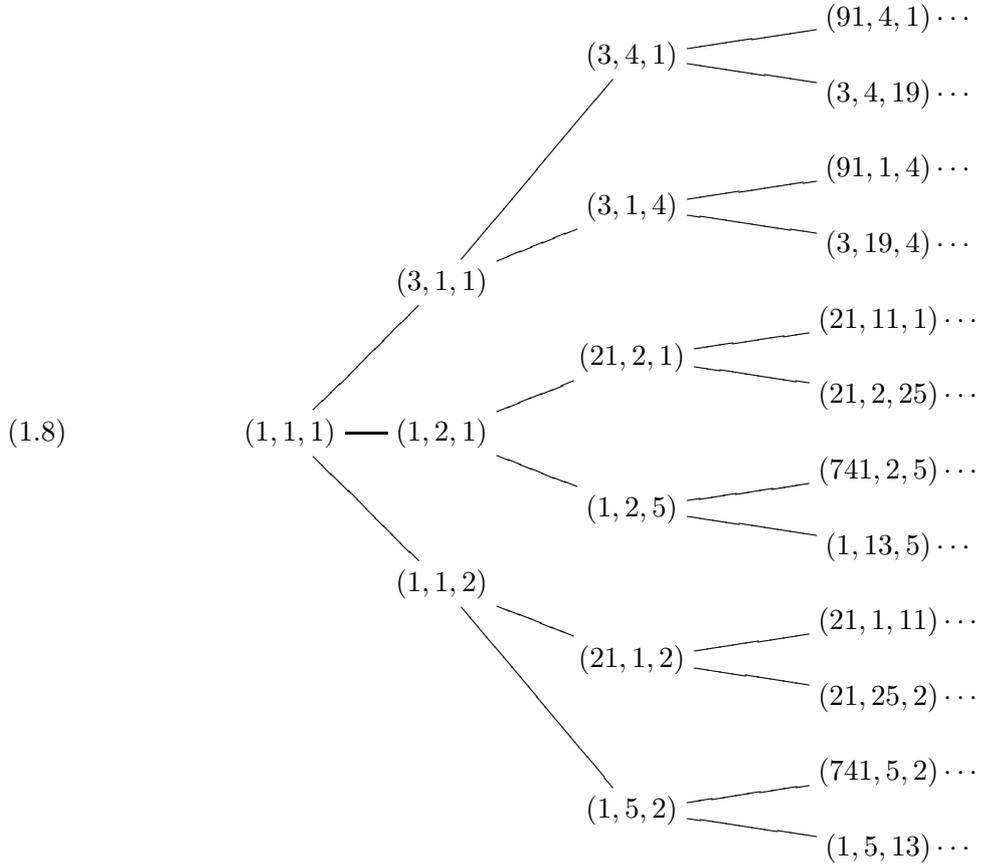
\begin{align}\label{tree2}
\begin{xy}(0,0)*+{(1,1,1)}="0",(20,20)*+{(3,1,1)}="1",(20,0)*+{(1,2,1)}="1'",(20,-20)*+{(1,1,2)}="1''",(45,50)*+{(3,4,1)}="20",(45,30)*+{(3,1,4)}="21",(45,10)*+{(21,2,1)}="22",(45,-10)*+{(1,2,5)}="23",(45,-30)*+{(21,1,2)}="24",(45,-50)*+{(1,5,2)}="25",(80,55)*+{(91,4,1)\cdots}="40",(80,45)*+{(3,4,19)\cdots}="41", (80,35)*+{(91,1,4)\cdots}="42", (80,25)*+{(3,19,4)\cdots}="43", (80,15)*+{(21,11,1)\cdots}="44", (80,5)*+{(21,2,25)\cdots}="45", (80,-5)*+{(741,2,5)\cdots}="46", (80,-15)*+{(1,13,5)\cdots}="47", (80,-25)*+{(21,1,11)\cdots}="48", (80,-35)*+{(21,25,2)\cdots}="49", (80,-45)*+{(741,5,2)\cdots}="410", (80,-55)*+{(1,5,13)\cdots}="411", \ar@{-}"0";"1"\ar@{-}"0";"1'"\ar@{-}"0";"1''"\ar@{-}"1";"20"\ar@{-}"1";"21"\ar@{-}"1'";"22"\ar@{-}"1'";"23"\ar@{-}"1''";"24"\ar@{-}"1''";"25"\ar@{-}"20";"40"\ar@{-}"20";"41"\ar@{-}"21";"42"\ar@{-}"21";"43"\ar@{-}"22";"44"\ar@{-}"22";"45"\ar@{-}"23";"46"\ar@{-}"23";"47"\ar@{-}"24";"48"\ar@{-}"24";"49"\ar@{-}"25";"410"\ar@{-}"25";"411"
\end{xy}.
\end{align}

The second main result is the following theorem:

\begin{theorem}\label{Diophantinetheorem2}
Every positive integer solution to \eqref{Diophantine4} appears exactly once in $\mathbb T^k$.
\end{theorem}
This is a generalization of Lampe's result \cite{lampe}*{Theorem 2.6}. In Lampe's paper, it is showed the case $k=0$ as mentioned above. To prove it, the $k_1=0,k_2=k_3=2$ case of Theorem \ref{Diophantinetheorem} is used in his paper. In the proof of Theorem \ref{Diophantinetheorem2}, we use the $k_1=k,k_2=k_3=2$ case of Theorem \ref{Diophantinetheorem}, which is its generalization.

The methods of enumerating the positive integer solutions to the equations mentioned above have one thing in common: it has a structure that can generate three another positive integer solutions from one positive integer solution. This operation is called the \emph{Vieta jumping} and is the key operation of the two main theorems given in this paper. This paper also explains that these Vieta jumpings and positive integer solutions have a structure derived from \emph{cluster algebra theory}. Cluster algebra is a class of commutative algebra generated by \emph{cluster variables}, introduced by \cite{fzi} in early 2000. The cluster variables refer to the all variables obtained sequentially by an oparation called  \emph{mutation}, and it is known that the combinatorial structure of the cluster variable and the mutation appears in various fields such as Teichm{\"u}ller theory, Poisson geometry, representation theory of quiver, gauge theory, knot theory, etc. Cluster algebra theory is also closely related to number theory. Immediately after the birth of cluster algebra, it is shown that the integer sequences called \emph{Somos-4} and \emph{Somos-5} can be seen as cluster variables, and the recurrence formula that gives it can be seen as a specialization of mutation (see \cite{fzl}). In the context of the Diophantine problem, it was first known that Vieta jumpings and positive integer solutions of the Markov Diophantine equation \eqref{Diophantine2} are a specialization of a class of mutations and cluster variables (for example, there is a description of it in \cite{fzw}), and then it was found by \cite{lampe} that those of equation \eqref{Diophantine8} are given by a specialization of another class of mutation and cluster variables. Recently, it was found by \cite{gyo21} that those of the equation \eqref{Diophantine3} can be given as a specialization of the mutation and cluster variable associated with the \emph{generalized cluster algebra}. In this paper, we will discuss the generalized cluster algebra structure of equations \eqref{Diophantine} and \eqref{Diophantine4}, including all of the above mentioned.

At the end of this paper, we will consider whether there are any other Diophantine equations with the structure of a generalized cluster algebra like these equations.

\subsection*{Acknowledgments}
The authors would like to thank Tomoki Nakanishi for checking their paper. The authors would like to thank Michael Somos for helpful comments. The authors would like to thank Matvey Zhuravlev for pointing out and correcting a mistake in the proof of Proposition \ref{nonsingular-induction}. This work was supported by JSPS KAKENHI Grant number JP20J12675.
\section{Proof of Theorem \ref{Diophantinetheorem} and its corollaries}
We will prove the first main theorem, Theorem \ref{Diophantinetheorem}. We begin with the following proposition:
\begin{proposition}\label{inductive-solution}
If $(x,y,z)=(a,b,c)$ is a positive integer solution to \eqref{Diophantine}, then so are $\left(\dfrac{b^2+k_1bc+c^2}{a},b,c\right),\left(a,\dfrac{a^2+k_2ac+c^2}{b},c\right)$, and $\left(a,b,\dfrac{a^2+k_3ab+b^2}{c}\right)$.
\end{proposition}
\begin{proof}
We prove only that $\left(\dfrac{b^2+k_1bc+c^2}{a},b,c\right)$ is a positive solution. The positivity is clear. We will show that $\left(\dfrac{b^2+k_1bc+c^2}{a},b,c\right)$ is an integer solution to \eqref{Diophantine}. Since $(a,b,c)$ is a solution to \eqref{Diophantine}, we have 
\[\dfrac{b^2+k_1bc+c^2}{a}=(3+k_1+k_2+k_3)bc-a-k_3b-k_2c.\]
Therefore, $\left(\dfrac{b^2+k_1bc+c^2}{a},b,c\right)$ is an integer triplet. For the sake of visibility, we set \[A:=(3+k_1+k_2+k_3)bc-a-k_3b-k_2c.\]
We will show $(A,b,c)$ is a solution to \eqref{Diophantine}. 
The sum and the product of $a$ and $A$ are
\begin{align*}
     &a + A = (3 + k_1 + k_2 + k_3)bc - k_3 b - k_2 c, \\
     &a \cdot A = b^2 + k_1 bc + c^2, 
\end{align*} 
respectively. By Vieta's formula, $a$ and $A$ are the solutions of the quadratic equation
\begin{align*}
    X^2 - \{(3+ k_1 + k_2 + k_3)bc - k_3 b - k_2 c\} X + b^2 + k_1 bc + c^2 =0. 
\end{align*}
Substituting $X=A$ and transfering some terms to the other side, we have
\begin{align*}
   A^2 + b^2 + c^2 + k_3 A b + k_1 b c + k_2 c A=(3 + k_3 + k_1 + k_2) A b c.
\end{align*}
This is an equality substituting $x=A=\dfrac{b^2+k_1bc+c^2}{a},y=b,z=c$ in \eqref{Diophantine}. 
\end{proof}
We call the operation $(a,b,c)\mapsto \left(\dfrac{b^2+k_1bc+c^2}{a},b,c\right)$ the \emph{first Vieta jumping}, $(a,b,c)\mapsto\left(a,\dfrac{a^2+k_2ac+c^2}{b},c\right)$ the \emph{second Vieta jumping}, and $(a,b,c)\mapsto \left(a,b,\dfrac{a^2+k_3ab+b^2}{c}\right)$ the \emph{third Vieta jumping}. We note that the Vieta jumping are involutions. The next step is to determine the solutions that contain two or more of the same number.
\begin{lemma}\label{singular}
In the positive integer solutions to \eqref{Diophantine}, the only solutions that contain repeated numbers are $(1,1,1),(k_1+2,1,1),(1,k_2+2,1)$ and $(1,1,k_3+2)$.
\end{lemma}
\begin{proof}
Let $(a,b,c)$ be a positive integer solution to \eqref{Diophantine} that contains repeated numbers. We prove the case of $a=b$. Then, by substituting $(a,a,c)$ for $(x,y,z)$ in \eqref{Diophantine}, we have
\[
(2+k_3)a^2+c^2+(k_1+k_2)ac=(3+k_1+k_2+k_3)a^2c.
\]
Therefore, we have 
\[
 c= \dfrac{a^2 k_3 + a^2 k_1 + a^2 k_2 + 3 a^2 - a k_1 - a k_2 \pm a\sqrt{ (a k_3 + (a - 1) k_1 + (a-1) k_2 + 3 a )^2 - 4 (k_3 + 2)}}{2}.
\]
We set $k=a k_3 + (a - 1) k_1 + (a-1) k_2 + 3 a>0$. In order for $c$ to be an integer, the inside of the square root must be a square number. Therefore, there exists an positive integer $l$ such that $l^2=k^2-4(k_3+2)$. Since $a\geq 1$, we have $k\geq k_3+3$. Therefore, $k+l>k_3+2$ holds. Since $(k+l)(k-l)=4(k_3+2)$, we have $1\leq k-l \leq 3$, and $(k-l,k+l)$ must be one of $(1,4(k_3+2)),(2,2(k_3+2)),\left(3,\dfrac{4(k_3+2)}{3}\right)$. Of the three, it cannot be $(1,4(k_3+2))$ and $\left(3,\dfrac{4(k_3+2)}{3}\right)$ because $k=\dfrac{(k+l)+(k-l)}{2}$ is an integer. When $(k-l,k+l)=(2, 2(k_3+2))$, we have $k=k_3+3$ and $l=k_3+1$. Thus we have $(a,a,c)=(1,1,1)$ or $(1,1,k_3+2)$. The cases that $a = c$ and $b = c$ can be proved in the same way.
\end{proof}
The triples $(1,1,1),(k_1+2,1,1),(1,k_2+2,1)$ and $(1,1,k_3+2)$ are said to be \emph{singular}, and other positive integer solutions to \eqref{Diophantine} are said to be \emph{nonsingular}.
\begin{proposition}\label{nonsingular-induction}
Let $(x,y,z)=(a,b,c)$ be a nonsingular positive integer solution to \eqref{Diophantine}, and we assume $a>b>c$. Then we have\vspace{3mm}
\begin{itemize}
\item[(1)] $\dfrac{a^2+k_2ac+c^2}{b}>a(>c)$\vspace{3mm},
\item[(2)] $\dfrac{a^2+k_3ab+b^2}{c}>a(>b)$\vspace{3mm},
\item[(3)] $b>\dfrac{b^2+k_1bc+c^2}{a}$.
\end{itemize}
\end{proposition}
\begin{proof}
We prove (1). We have
\begin{align*}
   \dfrac{a^2+k_2ac+c^2}{b}-a=\dfrac{a^2+k_2ac+c^2-ab}{b}>\dfrac{a^2+k_2ac+c^2-a^2}{b}=\dfrac{c^2+k_2ac}{b}>0.
\end{align*}
We can show (2) in the same way as (1). We will show (3).
We set
\begin{align*}
    f(x)&:=(x-a)\left(x-\dfrac{b^2+k_1bc+c^2}{a}\right)\\
    &=x^2-((3+k_1+k_2+k_3)bc-k_3b-k_2c)x+(b^2+c^2+k_1bc)
\end{align*}
It suffices to show that 
\[f(b)=(2+k_3)b^2-(3+k_1+k_2+k_3)b^2c+(k_1+k_2)bc+c^2<0.\] We consider a function from $\mathbb{R}^2$ to $\mathbb{R}$
\begin{align*}\label{equation}
    g(y,z)=(2+k_3)y^2-(3+k_1+k_2+k_3)y^2z+(k_1+k_2)yz+z^2.
\end{align*}
We remark that $g(b,c)=f(b)$. By considering the partial derivative of $g$ in the $y$ direction, we have
\[\dfrac{\partial g}{\partial y}=2(2+k_3)y-2(3+k_1+k_2+k_3)yz+(k_1+k_2)z.\]
When $y>z\geq 1$, we have
\begin{align*}
\dfrac{\partial g}{\partial y}(y,z)&<2(2+k_3)y-2(3+k_1+k_2+k_3)yz+(k_1+k_2)y\\
&=-y((6z-4)+k_3(2z-2)+k_1(2z-1)+k_2(2z-1))\\
&<-((6z-4)+k_3(2z-2)+k_1(2z-1)+k_2(2z-1))<0.
\end{align*}
Moreover, by considering the partial derivative of $g$ in the $z$ direction, we have
\[\dfrac{\partial g}{\partial z}=-(3+k_1+k_2+k_3)y^2+(k_1+k_2)y+2z.\]
When $y>z\geq 1$, we have
\begin{align*}
\dfrac{\partial g}{\partial z}(y,z)&<-(3+k_1+k_2+k_3)y^2+(k_1+k_2)y+2y\\
&<-(3+k_1+k_2+k_3)y^2+(k_1+k_2)y^2+2y^2<-y^2(1+k_3)<0.
\end{align*}
Therefore, $g(y,z)$ is strictly monotonically decreasing in the $y$ and $z$ directions in the range $y>z\geq 1$. Since $g(1,1)=0$, we have $g(b,c)=f(b)<0$.
\end{proof}
In Proposition \ref{nonsingular-induction}, we assume that $a>b> c$, but this assumption is not essential:

\begin{corollary}\label{nonsingular-induction-general}
Let $(x,y,z)=(a,b,c)$ be a nonsingular positive integer solution to \eqref{Diophantine}. We set $(a',b,c)$ (resp. $(a,b',c),(a,b,c')$) as the first (resp. second, third) Vieta jumping.
\begin{itemize}
    \item [(1)] If $a$ is maximal in $(a,b,c)$, then $a'$ is not maximal in $(a',b,c)$, $b'$ is maximal in $(a,b',c)$, and $c'$ is maximal in $(a,b,c')$,
    \item [(2)] if $b$ is maximal in $(a,b,c)$, then $a'$ is maximal in $(a',b,c)$, $b'$ is not maximal in $(a,b',c)$, and $c'$ is maximal in $(a,b,c')$,
    \item [(3)] if $c$ is maximal in $(a,b,c)$, then $a'$ is maximal in $(a',b,c)$, $b'$ is maximal in $(a,b',c)$, and $c'$ is not maximal in $(a,b,c')$.
\end{itemize}
\end{corollary}
\begin{proof}
When $a> b> c$, it is proved by Proposition \ref{nonsingular-induction}. The other cases are proved in the same way as the proof of Proposition \ref{nonsingular-induction}.
\end{proof}

\begin{remark}\label{remark-adjacent}
By Corollary \ref{nonsingular-induction-general}, for a nonsingular triplet $(a,b,c)$ in $\TT^{k_1,k_2,k_3}$,
\begin{itemize}
     \item [(i)] if $a$ is a maximal number in $(a,b,c)$, then the parent of $(a,b,c)$ is $\left(\dfrac{b^2+k_1bc+c^2}{a},b,c\right)$,
        
    \item [(ii)] if $b$ is a maximal number in $(a,b,c)$, then the parent of $(a,b,c)$ is $\left(a,\dfrac{a^2+k_2ac+c^2}{b},c\right)$,
    
    \item [(iii)] if $c$ is a maximal number in $(a,b,c)$, then the parent of $(a,b,c)$ is $\left(a,b,\dfrac{a^2+k_3ab+b^2}{c}\right)$.
\end{itemize}
Moreover, each non-singular triplet in $\mathbb T^{k_1,k_2,k_3}$ have a smaller maximum than its children. Therefore, singular triplets cannot be children of non-singular triplets in $\mathbb T^{k_1,k_2,k_3}$. Hence, each singular triplet appears in $\mathbb T^{k_1,k_2,k_3}$ once, and the above facts (i),(ii),(iii) are also true for singular triplets other than $(1,1,1)$. Thus, in $\TT^{k_1,k_2,k_3}$, the three vertices adjacent to each vertex $(a, b, c)$ are respectively the one where $a$ in $(a, b, c)$ is replaced by another number, the one where $b$ in $(a, b, c)$ is replaced by another number, and  the one where $c$ in $(a, b, c)$ is replaced by another number.
\end{remark}

Now, we will show Theorem \ref{Diophantinetheorem}.

\begin{proof}[Proof of Theorem \ref{Diophantinetheorem}]
By Proposition \ref{inductive-solution} and the fact that $(x,y,z)=(1,1,1)$ is a positive integer solution to \eqref{Diophantine}, all vertices in $\mathbb T^{k_1,k_2,k_3}$ are positive integer solutions to \eqref{Diophantine}. Suppose that $(x,y,z)=(a,b,c)$ is a nonsingular positive integer solution to \eqref{Diophantine}. Then, by Corollary \ref{nonsingular-induction-general}, there is one of the Vieta jumpings of $(a,b,c)$ whose maximal number is smaller than that of $(a,b,c)$. This process can be continued as long as the solution is nonsingular. Since the solutions that appear in this operation are always positive integer solutions, a singular solution will appear in a finite number of the operations. By Lemma \ref{singular}, when a nonsingular solution changes to a singular solution, the singular solution is $(k_1+2,1,1),(1,k_2+2,1)$ or $(1,1,k_3+2)$. Since any Vieta jumping of a triplet in $\TT^{k_1,k_2,k_3}$ is again in $\TT^{k_1,k_2,k_3}$ by Remark \ref{remark-adjacent}, we see that $(a,b,c)$ is contained in the vertices of the tree $\mathbb T^{k_1,k_2,k_3}$ by following above operations in reverse. We prove the uniqueness. If not, we see that $(1,1,1)$ is not unique by repeating above operations. This is a contradiction.
\end{proof}

As in the Markov Diophantine equation \eqref{Diophantine2} or the Gyoda's equation \eqref{Diophantine3}, there are several corollaries that can be established.

\begin{corollary}\label{relatively-prime}
For any positive integer solution $(x,y,z)=(a,b,c)$ to \eqref{Diophantine}, all pairs in $a,b,c$ are relatively prime.
\end{corollary}
\begin{proof}
The claim is true for $(a,b,c)=(1,1,1)$. 
We only prove that $a$ and $b$ are relatively prime. By transforming \eqref{Diophantine} as \[z^2=(3+k_1+k_2+k_3)xyz-x^2-y^2-k_1yz-k_2zx-k_3xy,\] and substituting $(x,y,z)=(a,b,c)$, if $a,b$ have a common divisor $d\neq 1$, then we see that $c$ can be divided by a prime divisor $d'$ of $d$. Thus, $d'$ is a common divisor of $a,b,c$. Therefore, by Proposition \ref{inductive-solution}, the neighbor $(a',b',c')$ of $(a,b,c)$ on the tree $\mathbb T^{k_1,k_2,k_3}$ whose maximal number is smaller than $\max\{a,b,c\}$ has the common divisor $d'$. By repeating this operation, we see that $d'$ is a common divisor of $(1,1,1)$. Thus, we must $d'=1$. This is a contradiction. Therefore, we have $d=1$.
\end{proof}

\begin{corollary}\label{cor}
Every number appearing in the tree $\mathbb T^{k_1,k_2,k_3}$ appears as the maximal number of some positive integer solution to \eqref{Diophantine}.
\end{corollary}

\begin{proof}
Let $n$ be a number appearing in $\mathbb T^{k_1,k_2,k_3}$. When $n=1,k_1+2,k_2+2,k_3+2$, they are the maximal numbers of $(1,1,1), (k_1+2,1,1),(1,k_2+2,1)$ and $(1,1,k_3+2)$, respectively. We assume $n\neq 1, k_1+2,k_2+2,k_3+2$. We take a positive integer solution $(x,y,z)=(a,b,c)$ containing $n$. We assume that $a> b>c$. If $n=a$, then we are done. If $n=b$, then $n$ is the maximal number in the neighbor of $(a,b,c)$ in the tree $\mathbb T^{k_1,k_2,k_3}$ obtained by swapping $a$ by Proposition \ref{inductive-solution}. If $n=c$, as we traverse the neighbors with smaller maximal number, $n$ becomes the second largest. Therefore, this case is attributed to the $n=b$ case. Even if the magnitude correlation of $a, b$, and $c$ are different, the statement is proved in the same way.
\end{proof}
\begin{remark}
In the Markov case, that is, $k_1=k_2=k_3=0$, there is a conjecture that triplets with a common maximum will coincide if the order of the components is reordered (the \emph{Markov conjecture}). However, in the general case, there can be essentially different triplets with a common maximum (i.e., they will not coincide even if the order of the components is reordered). Actually, when $k_1=1,k_2=2,k_3=0$, $(1,81,17)$ and $(7,81,2)$ are both solutions to \eqref{Diophantine}, as seen in \eqref{tree}. The Markov conjecture is proved to be true when the largest number in a triplet can be written as $p^n$ using the prime number $p$ (\cites{sch, but}), but this counterexample shows that even that does not hold in the general case.
\end{remark}
Let us consider the case of $k_1=k_2=k_3=2$, that is, the equation
\begin{align}\label{Diophantine5}
    x^2+y^2+z^2+2yz+2zx+2xy=9xyz.
\end{align}
In this situation, we have the following theorem:

\begin{theorem}
If a positive integer triplet $(a,b,c)$ is one of the solutions to the Markov equation \eqref{Diophantine2}, then $(a^2,b^2,c^2)$ is one of the solutions to \eqref{Diophantine5}. Conversely, if a positive integer triplet $(A,B,C)$ is one of the solutions to \eqref{Diophantine5}, then $(\sqrt{A},\sqrt{B},\sqrt{C})$ is one of the positive integer solutions to \eqref{Diophantine2}.
\end{theorem}

\begin{proof}
We prove the former statement. When $(a,b,c)=(1,1,1)$, it is clear. We assume that $(a^2,b^2,c^2)$ is an integer solution to \eqref{Diophantine5}. It suffices to show that the Vieta jumpings of $(a^2,b^2,c^2)$ in \eqref{Diophantine5} are given by \[\left(\left(\dfrac{b^2+c^2}{a}\right)^2,b^2,c^2\right), \left(a^2,\left(\dfrac{a^2+c^2}{b}\right)^2,c^2\right),\left(a^2,b^2,\left(\dfrac{a^2+b^2}{c}\right)^2\right).\] We only prove the case of the first Vieta jumping. The first Vieta jumping of $(a^2,b^2,c^2)$ in \eqref{Diophantine5} is 
\[\left(\dfrac{(b^2)^2+2b^2c^2+(c^2)^2}{a^2},b^2,c^2\right)=\left(\left(\dfrac{b^2+c^2}{a}\right)^2,b^2,c^2\right),\]
as desired. We will show the latter statement. By Theorem \ref{Diophantinetheorem}, each positive solution to \eqref{Diophantine5} has the form $(a^2,b^2,c^2)$, where $(a,b,c)$ is a solution to \eqref{Diophantine2}. This finishes the proof.
\end{proof}

\section{Proof of Theorem \ref{Diophantinetheorem2} and its corollaries}

To prove Theorem \ref{Diophantinetheorem2}, we consider the following equation:

\begin{align}\label{Diophantine6}
    X^2+Y^2+Z^2+kYZ+2ZX+2XY=(7+k)XYZ.
\end{align}
This is the equation substituted \eqref{Diophantine4} with $X=x, Y=y^2,Z=z^2$ and a specialization of \eqref{Diophantine} with $k_1=k,k_2=2,k_3=2$. Then, the Vieta jumpings of $(A,B,C)$ are 
\[\left(\dfrac{B^2+kBC+C^2}{A},B,C\right),\left(A,\dfrac{(A+C)^2}{B},C\right),\left(A,B,\dfrac{(A+B)^2}{C}\right).\] 
\begin{lemma}\label{specialization2k2}
Every positive integer solution to \eqref{Diophantine6} appears exactly once in $\mathbb T_{k,2,2}$. Moreover, for any positive integer solution $(A,B,C)$ to \eqref{Diophantine6}, there exist positive integers $b$ and $c$ such that $b^2=B$ and $c^2=C$.
\end{lemma}

\begin{proof}
The former statement follows from Theorem \ref{Diophantinetheorem}. We prove the latter statement. When $(X,Y,Z)=(1,1,1)$, it is clear. We assume that $(A,b^2,c^2)$ is a solution to  \eqref{Diophantine6}. The second Vieta jumping of $(A,b^2,c^2)$ in \eqref{Diophantine6} is $\left(A,\left(\dfrac{A+c^2}{b}\right)^2,c^2\right)$. Since $\left(\dfrac{A+c^2}{b}\right)^2$ is an integer, so is $\dfrac{A+c^2}{b}$. In the same way, we obtain $\left(A,b^2,\left(\dfrac{A+b^2}{c}\right)^2\right)$ from $(A,b^2,c^2)$ by the third Vieta jumping. These facts finish the proof. 
\end{proof}

\begin{proposition}\label{equivalent-condition}
If a positive integer triplet $(a,b,c)$ is one of the solutions to \eqref{Diophantine6}, then $(a,b^2,c^2)$ is one of the solutions to \eqref{Diophantine4}. Conversely, if a positive integer triplet $(A,B,C)$ is one of the solutions to \eqref{Diophantine4}, then $(A,\sqrt{B},\sqrt{C})$ is one of the positive integer solutions to \eqref{Diophantine6}.
\end{proposition}

\begin{proof}
The former statement is clear. The latter follows from Lemma \ref{specialization2k2}.  
\end{proof}

Now, we prove Theorem \ref{Diophantinetheorem2}.

\begin{proof}[Proof of Theorem \ref{Diophantinetheorem2}]
By Lemma \ref{specialization2k2} and Proposition \ref{equivalent-condition}, all positive integer solutions to \eqref{Diophantine4} are obtained from $(1,1,1)$ by repeating the Vieta jumpings 
\[(a,b,c)\mapsto \left(\dfrac{b^4+kb^2c^2+c^4}{a},b,c\right), (a,b,c)\mapsto \left(a,\dfrac{a+c^2}{b},c\right), (a,b,c)\mapsto \left(a,b,\dfrac{a+b^2}{c}\right).\]
Since the three vertices adjacent to $(a,b,c)$ in $\TT^k$ are triplets replacing different components of $(a,b,c)$, respectively, as in $\TT^{k,2,2}$, any Vieta jumping of a triplet in $\TT ^ k$ is again in $\TT ^ k$. Therefore, all positive integer solutions to \eqref{Diophantine4} appear in $\TT^k$. The uniqueness follows from the uniqueness of any triplet in $\TT^{k,2,2}$. Thus we obtain Theorem \ref{Diophantinetheorem2}.
\end{proof}
Next, we prove an analogue of Corollary \ref{relatively-prime}.
\begin{corollary}
For any positive integer solution $(x,y,z)=(a,b,c)$ to \eqref{Diophantine4}, all pairs in $a,b,c$ are relatively prime.
\end{corollary}

\begin{proof}
By Corollary \ref{relatively-prime}, $(a,b^2,c^2)$ is relatively prime. Thus $(a,b,c)$ is relatively prime.
\end{proof}

\begin{remark}
An analogue of Corollary \ref{cor} does not hold in \eqref{Diophantine4}. Actually, when $k=1$, $11$ appears in some positive integer solutions in \eqref{Diophantine4} (for example, $(21,11,1)$ is one of solutions). However, $11$ is not maximal number in any solutions containing it. See \eqref{tree2}.
\end{remark}

\section{Generalized cluster pattern of positive integer solutions}
In this section, we introduce a generalized cluster pattern and verify that the trees $\TT^{k_1,k_2,k_3}$ and $\TT^k$ in Section 1 are realized as special cases of generalized cluster patterns. We start with recalling definitions of seed mutations and generalized cluster patterns according to \cites{chsh,nak15,nakrup}\footnote{In \cites{chsh,nak15,nakrup}, seeds and their mutations are defined in a version with $y$-variables (coefficients), but in this paper, these are not necessary and have been omitted}.
Let $n\in \ZZ_{\geq1}$ and $\Fcal$ be a rational function field of $n$ indeterminates.
A \emph{labeled seed} is a triplet $(\mathbf{x},B,\mathbf{Z})$, where
\begin{itemize}
\item $\mathbf{x}=(x_1, \dots, x_n)$ is an $n$-tuple of elements of $\mathcal F$ forming a free generating set of $\mathcal F$,
\item $B=(b_{ij})$ is an $n \times n$ integer matrix which is \emph{skew-symmetrizable}, that is, there exists a positive integer diagonal matrix $S$ such that $SB$ is skew-symmetric. We call $S$ a \emph{skew-symmetrizer} of $B$,
\item $\mathbf{Z}=(Z_1,\dots,Z_n)$ is an $n$-tuple of polynomials with the coefficient in $\ZZ_{\geq 0}$
\[Z_i(u)=z_{i,0}+z_{i,1}u+\cdots+z_{i,d_i}u^{d_i}\]
satisfying $z_{i,0}=z_{i,d_i}=1$. 
\end{itemize}

We say that $\xx$ is a \emph{cluster}, and  we refer to $x_i$, $B$ and $Z_i$ as the \emph{cluster variable}, the \emph{exchange matrix} and the \emph{exchange polynomial}, respectively. Furthermore, we set $D=\text{diag} (d_1,\dots, d_n)$, that is a positive integer diagonal matrix of rank $n$.

For an integer $b$, we use the notation $[b]_+=\max(b,0)$. 
Let $(\mathbf{x}, B,\mathbf{Z})$ be a labeled seed, and let $k \in\{1,\dots, n\}$. The \emph{seed mutation $\mu_k$ in direction $k$} transforms $(\mathbf{x},B,\mathbf{Z})$ into another labeled seed $\mu_k(\mathbf{x}, B,\mathbf{Z})=(\mathbf{x'}, B',\mathbf{Z}')$ defined as follows:
\begin{itemize}
\item The entries of $B'=(b'_{ij})$ are given by
\begin{align} \label{eq:matrix-mutation}
b'_{ij}=\begin{cases}-b_{ij} &\text{if $i=k$ or $j=k$,} \\
b_{ij}+d_k\left(\left[ b_{ik}\right] _{+}b_{kj}+b_{ik}\left[ -b_{kj}\right]_+\right) &\text{otherwise.}
\end{cases}
\end{align}
\item The cluster variables $\mathbf{x'}=(x'_1, \dots, x'_n)$ are given by
\begin{align}\label{eq:x-mutation}
x'_j=\begin{cases}\dfrac{\left(\mathop{\prod}\limits_{i=1}^{n} x_i^{[-b_{ik}]_+}\right)^{d_k}Z_k\left(\mathop{\prod}\limits_{i=1}^{n} x_i^{b_{ik}}\right)}{x_k} &\text{if $j=k$,}\\
x_j &\text{otherwise.}
\end{cases}
\end{align}
\item The exchange polynomials $\mathbf{Z'}=(Z'_1, \dots, Z'_n)$ are given by
\begin{align}\label{eq:Z-mutation}
Z'_j(u)=\begin{cases}u^{d_k}Z_k(u^{-1})  &\text{if $j=k$,}\\
Z_j(u) &\text{otherwise.}
\end{cases}
\end{align}
\end{itemize}

Let $\mathbb{T}_n$ be the \emph{$n$-regular tree} whose edges are labeled by the numbers $1, \dots, n$ such that the $n$ edges emanating from each vertex have different labels. We write
$\begin{xy}(0,0)*+{t}="A",(10,0)*+{t'}="B",\ar@{-}^k"A";"B" \end{xy}$
to indicate that vertices $t,t'\in \mathbb{T}_n$ are joined by an edge labeled by $k$. We fix an arbitrary vertex $t_0\in \TT_n$, which is called the \emph{rooted vertex}.
A \emph{generalized cluster pattern} is an assignment of a labeled seed $\Sigma_t=(\xx_t,B_t,\mathbf Z_t)$ to every vertex $t\in \mathbb{T}_n$ such that the labeled seeds $\Sigma_t$ and $\Sigma_{t'}$ assigned to the endpoints of any edge
$\begin{xy}(0,0)*+{t}="A",(10,0)*+{t'}="B",\ar@{-}^k"A";"B" \end{xy}$
are obtained from each other by the seed mutation in direction $k$. When the initial seed is $\Sigma_{t_0}=(\xx,B,\mathbf Z)$, we denote by $CP_{(\xx,B,\mathbf Z)}\colon t\mapsto \Sigma_t$ this assignment.
The degree $n$ of the regular tree $\TT_n$ is called the \emph{rank} of a generalized cluster pattern $CP_{(\xx,B,\mathbf Z)}$.
\begin{remark}
When $D=I_n$ (the identity matrix), a generalized cluster pattern coincides with a \emph{cluster pattern} defined in \cite{fziv}. 
\end{remark}
We will now see examples that give positive integer solutions to \eqref{Diophantine}. First, we consider the case of $k_1=k_2=k_3=0$. we set $B=\begin{bmatrix}0&2&-2\\-2&0&2\\2&-2&0\end{bmatrix}$ and $Z_1(u)=Z_2(u)=Z_3(u)=1+u$. For any $k\in\{1,2,3\}$, we have 
\begin{align*}
    \mu_k(\pm B)=\mp B, \quad\mu_k(\mathbf Z)=\mathbf Z.
\end{align*} Therefore, the mutation of cluster variables are 
\begin{align*}\label{eq:x-mutation2}
x'_j=\begin{cases}\dfrac{x_\ell^2(1+x_\ell^{-2}x_m^2)}{x_k}=\dfrac{x_\ell^2+x_m^2}{x_k} &\text{if $j=k$,}\\
x_j &\text{otherwise,}
\end{cases}
\end{align*}
where $\{k,\ell,m\}=\{1,2,3\}$.
This oparation coincides with the Vieta jumping of in the Markov Diophantine equation.

Next, we consider the case of $k_3\neq 0$ and $k_1=k_2=0$. we set $B=\begin{bmatrix}0&2&-1\\-2&0&1\\2&-2&0\end{bmatrix}$ and $Z_1(u)=Z_2(u)=1+u, Z_3(u)=1+k_3u+u^2$. Then, for any $k\in\{1,2,3\}$, we have 
\begin{align*}
    \mu_k(\pm B)=\mp B,\quad \mu_k(\mathbf Z)=\mathbf Z.
\end{align*}
Therefore, when $k=1,2$, the mutations of cluster variables are 
\begin{align*}
x'_j=\begin{cases}\dfrac{x_\ell^2(1+x_\ell^{-2}x_m^2)}{x_k}=\dfrac{x_\ell^2+x_m^2}{x_k} &\text{if $j=k$,}\\
x_j &\text{otherwise,}
\end{cases}
\end{align*}
where $\{k,\ell,m\}=\{1,2,3\}$. When $k=3$, we have
\begin{align*}
x'_j=\begin{cases}\dfrac{x_1^2(1+k_3x_1^{-1}x_2+x_1^{-2}x_2^{2})}{x_3}=\dfrac{x_1^2+k_3x_1x_2+x_2^2}{x_3} &\text{if $j=k=3$,}\\
x_j &\text{otherwise.}
\end{cases}
\end{align*}
This oparation coincides with the Vieta jumping in the equation
\[x^2+y^2+z^2+k_3xy=(3+k_3)xyz.\]
The Vieta jumping in each equations treated in this paper is described by the mutation of a generalized cluster pattern. See the following theorem:
\begin{theorem}\label{clusterstructure}
For each equation given in \eqref{Diophantine} and \eqref{Diophantine4}, $B$ and $\mathbf{Z}$ (and $D$) are set as in Table \ref{table}.
\begin{table}[ht]
\begin{center}
   \begin{equation*}
\begin{array}{|c|c|c|c|}
\hline
&&&\\[-4mm]
\mathrm{Equation} & B & \mathbf Z &D\\
\hline
&&&\\[-3mm]
x^2+y^2+z^2=3xyz &\begin{bmatrix}0&2&-2\\-2&0&2\\2&-2&0\end{bmatrix}& \begin{cases}Z_1(u)=1+u\\Z_2(u)=1+u\\Z_3(u)=1+u\end{cases}&\begin{bmatrix}1&0&0\\0&1&0\\0&0&1\end{bmatrix}\\[3mm]
\hline
&&&\\[-3mm]
x^2+y^2+z^2+k_3xy=(3+k_3)xyz&\begin{bmatrix}0&2&-1\\-2&0&1\\2&-2&0\end{bmatrix}&\begin{cases}Z_1(u)=1+u\\Z_2(u)=1+u\\Z_3(u)=1+k_3u+u^2\end{cases}&\begin{bmatrix}1&0&0\\0&1&0\\0&0&2\end{bmatrix}\\[1mm]
\hline
&&&\\[-3mm]
\begin{matrix}x^2+y^2+z^2+k_1yz+k_3xy\\=(3+k_1+k_3)xyz\end{matrix}&\begin{bmatrix}0&2&-1\\-1&0&1\\1&-2&0\end{bmatrix}& \begin{cases}Z_1(u)=1+k_1u+u^2\\Z_2(u)=1+u\\Z_3(u)=1+k_3u+u^2\end{cases} &\begin{bmatrix}2&0&0\\0&1&0\\0&0&2\end{bmatrix}\\[1mm]
\hline
&&&\\[-3mm]
\begin{matrix}x^2+y^2+z^2+k_1yz+k_2zx+k_3xy\\=(3+k_1+k_2+k_3)xyz\end{matrix}&\begin{bmatrix}0&1&-1\\-1&0&1\\1&-1&0\end{bmatrix}& \begin{cases}Z_1(u)=1+k_1u+u^2\\Z_2(u)=1+k_2u+u^2\\Z_3(u)=1+k_3u+u^2\end{cases}&\begin{bmatrix}2&0&0\\0&2&0\\0&0&2\end{bmatrix} \\[1mm]
\hline
&&&\\[-3mm]
x^2+y^4+z^4+2xy^2+2z^2x=7xy^2z^2&\begin{bmatrix}0&1&-1\\-4&0&2\\4&-2&0\end{bmatrix}& \begin{cases}Z_1(u)=1+u\\Z_2(u)=1+u\\Z_3(u)=1+u\end{cases}&\begin{bmatrix}1&0&0\\0&1&0\\0&0&1\end{bmatrix} \\[1mm]
\hline
&&&\\[-3mm]
\begin{matrix}
x^2+y^4+z^4+2xy^2+ky^2z^2+2z^2x\\=(7+k)xy^2z^2\end{matrix}&\begin{bmatrix}0&1&-1\\-2&0&2\\2&-2&0\end{bmatrix}& \begin{cases}Z_1(u)=1+ku+u^2\\Z_2(u)=1+u\\Z_3(u)=1+u\end{cases}&\begin{bmatrix}2&0&0\\0&1&0\\0&0&1\end{bmatrix}\\[1mm]
\hline
\end{array}
\end{equation*}
\caption{Equations and corresponding triplets $(B,\mathbf Z,D)$}\label{table}
\end{center}
\end{table}
Then, the generalized cluster pattern $CP_{(\xx,B,\mathbf Z)}$ with a substitution $x_1=x_2=x_3=1$ gives the tree $\mathbb T^{k_1,k_2,k_3}$ or $\TT^k$ giving positive integer solutions to the corresponding equation  (where we ignore exchange matrices in the cluster pattern and consider only clusters). 
\end{theorem}
This theorem allows us to say that the Vieta jumpings and positive integer solutions of equations \eqref{Diophantine} and \eqref{Diophantine4} have the structures of a generalized cluster patterns.

\begin{remark}
Of the equations in Table \ref{table}, the generalized cluster patterns corresponding to equations other than the Markov equation \eqref{Diophantine2} and Gyoda's equation \eqref{Diophantine3} are not treated in \cites{chsh,nak15,nakrup}. However, it can be seen that these patterns also satisfy the same properties as the generalized cluster patterns for the classes treated in these papers. See \cite{nak24} for details.    
\end{remark}

\section{Questions and consideration of class of rank 2}
The cluster patterns corresponding to the equations listed in Table \ref{table} satisfy the following two conditions:
\begin{condition}\label{condition}\indent
\begin{itemize}
    \item [(1)]The exchange matrix is multiplied by $-1$ for a mutation in any direction,
    \item [(2)]the exchange polynomials are mutation invariant.
\end{itemize}
\end{condition}
Moreover, these six triplets $(B,\mathbf Z,D)$ in the table in Theorem \ref{clusterstructure} can be divided into two types: for the top four in the Table \ref{table}, $BD=\begin{bmatrix}0&2&-2\\-2&0&2\\2&-2&0\end{bmatrix}$ is satisfied, and for the remaining two, $BD=\begin{bmatrix}0&1&-1\\-4&0&2\\4&-2&0\end{bmatrix}$ is satisfied.
At present, we know a cluster pattern that satisfies these two conditions and $BD=\begin{bmatrix}0&1&-1\\-4&0&2\\4&-2&0\end{bmatrix}$ but for which no corresponding equation has been found. It is the cluster pattern determined by 
\begin{align}\label{expectedpattern}
    B=\begin{bmatrix}0&1&-1\\-1&0&2\\1&-2&0\end{bmatrix} \text, \begin{cases}Z_1(u)=1+k_3u+k_1u^2+k_3u^3+u^4,\\Z_2(u)=1+u,\\Z_3(u)=1+u,\end{cases} D=\begin{bmatrix}4&0&0\\0&1&0\\0&0&1\end{bmatrix}.
    \end{align}
Therefore, the following question can be considered.

\begin{question}
\noindent
Is there a Diophantine equation corresponding to \eqref{expectedpattern}?
\end{question}
As a more general question, the following problems are considered.
\begin{question}\label{question}\noindent
\begin{itemize}
\item[(1)] What kind of a triplet $(B,\mathbf Z,D)$ satisfying the two condition in Condition \ref{condition} such that $BD$ is neither $\begin{bmatrix}0&2&-2\\-2&0&2\\2&-2&0\end{bmatrix}$ nor $\begin{bmatrix}0&1&-1\\-4&0&2\\4&-2&0\end{bmatrix}$?
\item[(2)] Is there a general way to construct a Diophantine equation from information in $(B,\mathbf Z,D)$?
\end{itemize}
\end{question}
We will now consider Question \ref{question}. All of cluster patterns of rank 2 satisfy Condition \ref{condition} (1). Therefore, there are infinitely many cluster patterns of rank 2 which is the answer to Question \ref{question} (1). In this class, there are cluster patterns whose corresponding equations are derived from Theorem \ref{clusterstructure}. Since all cluster patterns treated in Theorem 4.2 are of rank 3, each seed can be mutated in three directions. We consider prohibiting mutaions in one of these directions. By substituting 1 to the cluster variable corresponding to the direction in which the mutation was prohibited, the cluster pattern that was originally rank 3 can be viewed as that of rank 2. In this case, the exchange matrix corresponding to the cluster pattern is a submatrix of the original one that removes the row and column corresponding to the direction in which mutation is prohibited. The equation corresponding to this cluster pattern is the equation that substitute $1$ to the variable corresponding to the direction in which mutation is prohibited. Therefore, the following theorem holds.
\begin{theorem}\label{clusterstructurerank2}
We set equations, $B$ and $\mathbf{Z}$ (and $D$) as in Table \ref{table2}.
\begin{table}[ht]
   \begin{equation*}
\begin{array}{|c|c|c|c|}
\hline
&&&\\[-4mm]
\mathrm{Equation} & B & \mathbf Z &D\\
\hline
&&&\\[-3mm]
x^2+y^2+1=3xy &\begin{bmatrix}0&2\\-2&0\end{bmatrix}& \begin{cases}Z_1(u)=1+u\\Z_2(u)=1+u\end{cases}&\begin{bmatrix}1&0\\0&1\end{bmatrix}\\[3mm]
\hline
&&&\\[-3mm]
x^2+y^2+k_3x+1=(3+k_3)xy&\begin{bmatrix}0&1\\-2&0\end{bmatrix}&\begin{cases}Z_1(u)=1+u\\Z_2(u)=1+k_3u+u^2\end{cases}&\begin{bmatrix}1&0\\0&2\end{bmatrix}\\[1mm]
\hline
&&&\\[-3mm]
x^2+y^2+k_3x+k_1y+1=(3+k_3+k_1)xy&\begin{bmatrix}0&1\\-1&0\end{bmatrix}& \begin{cases}Z_1(u)=1+k_1u+u^2\\Z_2(u)=1+k_3u+u^2\end{cases} &\begin{bmatrix}2&0\\0&2\end{bmatrix}\\[1mm]
\hline
&&&\\[-3mm]
x^2+y^4+2x+1=5xy^2&\begin{bmatrix}0&1\\-4&0\end{bmatrix}& \begin{cases}Z_1(u)=1+u\\Z_2(u)=1+u\end{cases}&\begin{bmatrix}1&0\\0&1\end{bmatrix} \\[1mm]
\hline
&&&\\[-3mm]
x^2+y^4+ky^2+2x+1=(5+k)xy^2&\begin{bmatrix}0&1\\-2&0\end{bmatrix}& \begin{cases}Z_1(u)=1+ku+u^2\\Z_2(u)=1+u\end{cases}&\begin{bmatrix}2&0\\0&1\end{bmatrix}\\[1mm]
\hline
\end{array}
\end{equation*}
\caption{Equations and corresponding triplets $(B,\mathbf Z,D)$}\label{table2}
\end{table}
Then, the generalized cluster pattern $CP_{(\xx,B,\mathbf Z)}$ with a substitution $x_1=x_2=1$ gives the tree giving all positive integer solutions to the corresponding equation (where we ignore exchange matrices in the cluster pattern and consider only clusters). 
\end{theorem}

Apart from the cluster pattern induced by Theorem \ref{clusterstructure}, we give the equation induced by the cluster pattern of type $A_2$, i.e. the cluster pattern determined by
\begin{align}
    B=\begin{bmatrix}0&1\\-1&0\end{bmatrix} \text, \begin{cases}Z_1(u)=1+u,\\Z_2(u)=1+u,\end{cases}D=\begin{bmatrix}1&0\\0&1\end{bmatrix}.
    \end{align}
This cluster pattern differs from the ones we have considered above in that it has finitely many cluster variables. By computing the cluster variable according to the mutation rule
\[\mu_1(x_1,x_2)=\left(\dfrac{x_2+1}{x_1},x_2\right) \text{ and }\mu_2(x_1,x_2)=\left(x_1,\dfrac{x_1+1}{x_2}\right),\]
 we see that clusters in the cluster pattern of type $A_2$ are
\begin{align}
    &(x_1,x_2),\left(x_1,\dfrac{x_1+1}{x_2}\right),\left(\dfrac{x_1+x_2+1}{x_1x_2},\dfrac{x_1+1}{x_2}\right),\left(\dfrac{x_1+x_2+1}{x_1x_2},\dfrac{x_2+1}{x_1}\right), \left(x_2,\dfrac{x_2+1}{x_1}\right),\label{clusterA2}\\
    &(x_2,x_1),\left(\dfrac{x_1+1}{x_2},x_1\right),\left(\dfrac{x_1+1}{x_2},\dfrac{x_1+x_2+1}{x_1x_2}\right),\left(\dfrac{x_2+1}{x_1},\dfrac{x_1+x_2+1}{x_1x_2}\right), \left(\dfrac{x_2+1}{x_1},x_2\right)\nonumber
    \end{align}
in total. Therefore, we want to find an equation such that the five pairs\[(1,1),(1,2),(2,1),(2,3),(3,2)\] where $1$ is substituted for $x_1$ and $x_2$ of \eqref{clusterA2}, are all positive integer solutions.

The set consisting of cluster variables
\[\left\{x_1,x_2,\dfrac{x_2+1}{x_1},\dfrac{x_1+x_2+1}{x_1x_2},\dfrac{x_1+1}{x_2}\right\}\]
appearing in \eqref{clusterA2} (it is called the \emph{Lyness 5-cycle}) is invariant by the substitutions
\begin{align*}
 \left(x_1,x_2\right)\mapsto\left(\dfrac{x_2+1}{x_1},x_2\right)\text{ and } \left(x_1,x_2\right)\mapsto\left(x_1,\dfrac{x_1+1}{x_2}\right).
 \end{align*}
Therefore, we have the following proposition.
\begin{proposition}\label{symmetricpolynomial}
Let $f(a_1,a_2,a_3,a_4,a_5)$ be a symmetric polynomial of five variables. Then, the equation
\begin{align}\label{5variableequation}
f\left(x,y,\dfrac{y+1}{x},\dfrac{x+y+1}{xy},\dfrac{x+1}{y}\right)=f\left(x,y,\dfrac{y+1}{x},\dfrac{x+y+1}{xy},\dfrac{x+1}{y}\right)\Bigg|_{x=y=1}
\end{align}
has positive integer solutions
\[(x,y)=(1,1),(1,2),(2,1),(2,3),(3,2).\]
\end{proposition}
Furthermore, the following proposition also holds:
\begin{proposition}\label{allpositivesolutionsA2}
If $f(a_1,a_2,a_3,a_4,a_5)=a_1+a_2+a_3+a_4+a_5$ in Proposition \ref{symmetricpolynomial}, then all positive integer solutions to
\begin{align}\label{A2equation}
    x^2+y^2+2x+2y+x^2y+xy^2+1=9xy,
\end{align}
which is corresponding equation to \eqref{5variableequation}, are \[(x,y)=(1,1),(1,2),(2,1),(2,3),(3,2).\]
\end{proposition}
\begin{proof}
We set 
\[g(x,y)=x^2+y^2+2x+2y+x^2y+xy^2+1-9xy.\]
It suffice to show that positive integer pairs $(x,y)$ satisfying $g(x,y)=0$ are given by 
\[(x,y)=(1,1),(1,2),(2,1),(2,3),(3,2).\]
First, we will consider the case of $x\geq y$. If $x=y$, then we can see that the positive solution to $g(x,y)=0$ is only $(x,y)=(1.1)$ immediately. We assume that $x>y$. By considering the partial derivative of $g$ in $x$ direction, we have
\[\dfrac{\partial g}{\partial x}=y^2+2xy+2x+2-9y.\]
By using $x>y$, we have
\[\dfrac{\partial g}{\partial x}(x,y)>3y^2-7y+2.\]
Therefore, if $x>y\geq 3$, then we have $\dfrac{\partial g}{\partial x}(x,y)>0$. On the other hand, by considering the partial derivative of $g$ in $x$ direction, we have
\[\dfrac{\partial g}{\partial y}=x^2+2xy+2y+2-9x=\left(x-\dfrac{9}{2}\right)^2+2xy+2x-\dfrac{73}{4}.\]
Therefore, if $x>y\geq 3$, then we have $\dfrac{\partial g}{\partial y}(x,y)>0$. Now, since $g(4,3)=16>0$, we have $g(x,y)>0$ when $x$ and $y$ are integer and $x>y\geq 3$. Second, we will consider the case of $y> x$. By symmetry of $g(x,y)$ for $x$ and $y$, we have $g(x,y)>0$ if $x$ and $y$ are integer and $y>x\geq 3$. Therefore, the only possible pairs of integers that satisfy $g(x,y)=0$ are \[(1,1),(1,2),(1,3),(2,1),(2,3),(3,1),(3,2).\]
Of these, $g(3,1)=g(1,3)=4$, thus $(x,y)=(1,3),(3,1)$ are not solutions. The other five are all solutions according to Proposition \ref{symmetricpolynomial}.
\end{proof}
\begin{remark}
In Proposition \ref{allpositivesolutionsA2}, even though $f(a_1,a_2,a_3,a_4,a_5)=a_1a_2a_3a_4a_5$, the equation corresponding to \eqref{5variableequation} is the same as \eqref{A2equation}.
\end{remark}
From the above, the cluster pattern of rank 2 for which the corresponding equation is known is given in Table \ref{table3}.
\begin{table}[ht]
   \begin{equation*}
\begin{array}{|c|c|c|c|}
\hline
&&&\\[-4mm]
\mathrm{Equation} & B & \mathbf Z &D\\
\hline
&&&\\[-3mm]
x^2+y^2+2x+2y+x^2y+xy^2+1=9xy &\begin{bmatrix}0&1\\-1&0\end{bmatrix}& \begin{cases}Z_1(u)=1+u\\Z_2(u)=1+u\end{cases}&\begin{bmatrix}1&0\\0&1\end{bmatrix}\\[3mm]
\hline
&&&\\[-3mm]
x^2+y^2++k_1y+k_3x+1=(3+k_1+k_3)xy&\begin{bmatrix}0&1\\-1&0\end{bmatrix}& \begin{cases}Z_1(u)=1+k_1u+u^2\\Z_2(u)=1+k_3u+u^2\end{cases} &\begin{bmatrix}2&0\\0&2\end{bmatrix}\\[1mm]
\hline
&&&\\[-3mm]
x^2+y^2+k_3x+1=(3+k_3)xy&\begin{bmatrix}0&1\\-2&0\end{bmatrix}&\begin{cases}Z_1(u)=1+u\\Z_2(u)=1+k_3u+u^2\end{cases}&\begin{bmatrix}1&0\\0&2\end{bmatrix}\\[1mm]
\hline
&&&\\[-3mm]
x^2+y^4+ky^2+2x+1=(5+k)xy^2&\begin{bmatrix}0&1\\-2&0\end{bmatrix}& \begin{cases}Z_1(u)=1+ku+u^2\\Z_2(u)=1+u\end{cases}&\begin{bmatrix}2&0\\0&1\end{bmatrix}\\[1mm]
\hline
&&&\\[-3mm]
x^2+y^4+2x+1=5xy^2&\begin{bmatrix}0&1\\-4&0\end{bmatrix}& \begin{cases}Z_1(u)=1+u\\Z_2(u)=1+u\end{cases}&\begin{bmatrix}1&0\\0&1\end{bmatrix} \\[1mm]
\hline
&&&\\[-3mm]
x^2+y^2+1=3xy &\begin{bmatrix}0&2\\-2&0\end{bmatrix}& \begin{cases}Z_1(u)=1+u\\Z_2(u)=1+u\end{cases}&\begin{bmatrix}1&0\\0&1\end{bmatrix}\\[3mm]
\hline
\end{array}
\end{equation*}
\caption{Equations and corresponding triplets $(B,\mathbf Z,D)$}\label{table3}
\end{table}
In order to find the answer to Question \ref{question} (2) about the cluster patterns of rank 2, the first thing to do is to consider the following question:
\begin{question}
Are there any laws between the triplets $(B,\mathbf Z,D)$  and the equations given in Table \ref{table3}?
\end{question}
\bibliography{myrefs}
\end{document}